\documentclass[11pt,letterpaper,oneside]{amsart}

\usepackage[utf8]{inputenc}
\usepackage{amsthm}
\usepackage{amssymb}
\usepackage{amsmath, amscd}
\usepackage{mathrsfs}
\usepackage{enumitem}
\usepackage{mathtools}
\usepackage{bm} 
\usepackage{tcolorbox}
\usepackage{xcolor}
\usepackage[margin=1.3in]{geometry}
\newtheorem{theorem}{Theorem}[section]
\newtheorem*{theorem*}{Theorem}
\newtheorem{lemma}[theorem]{Lemma}
\newtheorem*{lemma*}{Lemma}

\newtheorem*{proposition*}{Proposition}
\newtheorem{corollary}[theorem]{Corollary}
\newtheorem*{corollary*}{Corollary}

\newtheorem*{conjecture*}{Conjecture}
\newtheorem*{claim*}{Claim}

\theoremstyle{definition}
\newtheorem{definition}[theorem]{Definition}
\newtheorem*{definition*}{Definition}
\newtheorem{propositionSp}[theorem]{Proposition}
\newtheorem*{propositionSp*}{Proposition}

\newtheorem*{corollarySp*}{Corollary}

\newcommand{\st}[0]{\mid}
\newcommand{\bdry}[0]{\partial}
\newcommand{\union}[0]{\cup}
\newcommand{\intersect}[0]{\cap}
\DeclareMathOperator{\diam}{diam}
\DeclarePairedDelimiter\set{\{}{\}}

\newcommand{\R}[0]{\mathbb{R}}

\title{Equivalent Topologies on the Contracting Boundary}
\author{Vivian He}

\begin{document}
	\maketitle
	\begin{abstract}
	The contracting boundary of a proper geodesic metric space generalizes the Gromov boundary of a hyperbolic space. It consists of contracting geodesics up to bounded Hausdorff distances. Another generalization of the Gromov boundary is the $\kappa$--Morse boundary with a sublinear function $\kappa$. The two generalizations model the Gromov boundary based on different characteristics of geodesics in Gromov hyperbolic spaces. It was suspected that the $\kappa$--Morse boundary contains the contracting boundary. We will prove this conjecture: when $\kappa =1$ is the constant function, the 1-Morse boundary and the contracting boundary are equivalent as topological spaces. 
	\end{abstract}

\section{Introduction}
There have been many attempts to construct a boundary on proper geodesic metric spaces in order to generalize the Gromov boundary on hyperbolic spaces. The Gromov boundary is the set of equivalence classes of geodesics up to bounded Hausdorff distances. In a similar fashion, Charney and Sultan \cite{CS} defined the \emph{contracting boundary} on CAT(0) spaces to be the set of contracting geodesics modulo Hausdorff equivalence. The \emph{contracting} property was introduced to imitate the behaviour of geodesics in hyperbolic spaces. Cashen and Mackay \cite{CM} generalized the contracting boundary to proper geodesic metric spaces and defined a topology which is quasi-isometry invariant and metrizable when the space is the Cayley graph of a finitely generated group.

More recently, Qing, Rafi, and Tiozzo \cite{QRT1,QRT2} defined the \emph{$\kappa$--Morse boundary} associated to a sublinear function $\kappa$. The topology on the $\kappa$--Morse boundary is quasi-isometry invariant and metrizable. When $\kappa \equiv 1$, the 1-Morse property on a quasi-geodesic is equivalent to the contracting property \cite[Theorem 2.2]{CM}. This suggests that $1$--Morse boundary and the contracting boundary are equal as sets. We will check that in detail in section 3. Moreover, we will show that the topologies are also equivalent. 

\begin{theorem}\label{main}
	For a proper geodesic metric space $X$, the 1-Morse boundary $\bdry_1 X$ and the contracting boundary $\bdry_c X$ are equivalent as topological spaces. 
\end{theorem}
The $\kappa$--Morse boundary is metrizable for any proper geodesic metric space $X$ and sublinear function $\kappa$ \cite[Theorem 4.10]{QRT2}. So we now have this immediate result:
\begin{corollary}
	The contracting boundary on a proper geodesic metric space is metrizable. 
\end{corollary}
This strengthens the result from \cite{CM} which shows the metrizability of the contracting boundary when $X$ is the Cayley graph of a finitely generated group.
\subsection*{Background}
The quasi-isometry invariance of the boundary is significant as it respects Cayley graphs of groups. If we copy the definition of the Gromov boundary to proper geodesic metric spaces then the quasi-isometry invariance no longer holds: Croke and Kleiner \cite{CK} constructed an example of two quasi-isometric CAT(0) spaces whose Gromov boundaries are not homeomorphic.

To fix this, Charney and Sultan \cite{CS} defined the Morse boundary of a complete CAT(0) space to be the subset of the Gromov boundary consisting of only Morse geodesics. A geodesic is $M$--Morse if any quasi-geodesic with end points on the geodesic is contained in the $M$-neighbourhood of the geodesic. Charney and Sultan showed that this boundary equipped with the direct limit topology is invariant under quasi-isometries. However, this topology is generally not first countable as shown in \cite{CM}. Cashen and Mackay \cite{CM} extended this definition to proper geodesic metric spaces and named it the contracting boundary. They introduced a topology on the contracting boundary which is first countable, Hausdorff, and regular. Moreover, they showed that when the space is the Cayley graph of a finitely generated group, the boundary is metrizable.

Qing, Rafi, and Tiozzo \cite{QRT1} further generalized this notion of the contracting boundary to the $\kappa$--Morse boundary. They did this by introducing a sublinear multiplicative error term on the Morse constant. This relaxation encapsulates the asymptotic behaviour of random walks on spaces \cite{QRT1}.

\section{Contracting Boundary}
In this section, we will restate the definition and notable properties of the contracting boundary as introduced in \cite{CM} and \cite{CS}. Let $X$ be a proper geodesic metric space with base point $\mathfrak{o}$ and metric $d_X$. We say a function $\rho: [0, \infty) \to [1, \infty)$ is sublinear if $\lim_{x\to \infty}\frac{\rho(x)}{x} = 0$. For simplicity, we will also require $\rho$ to be increasing and concave. The following definitions are equivalent. 

\begin{definition}\label{contracting}
	Let $Z$ be a closed subset of $X$ and $\pi_Z: X \to 2^Z$ be the closest point projection to $Z$. We say that $Z$ is contracting if there is a sublinear function $\rho_Z$ such that for all $x$ and $y$ in $X$, 
\begin{equation*}
	d(x,y) \leq d(x,Z) \qquad \implies \qquad  \diam (\pi_Z(x) \union \pi_Z(y)) \leq \rho(d(x,Z)).
\end{equation*}
\end{definition}

\begin{definition}\label{morse}
Let $Z$ be a closed subset of $X$, we say $Z$ is Morse if there exists a proper function $m_Z: [1,\infty) \times [0,\infty) \to \R$ for every $q \geq 1$ and every $Q \geq 0$, every $(q,Q)$-quasi-geodesic with endpoints in $Z$ is contained in the $m_Z(q,Q)$-neighbourhood of $Z$.
\end{definition}
The equivalence of contracting sets and Morse sets is proven in \cite{CM}. The notion of Morse is generalized to $\kappa$--Morse, which is then used to define the $\kappa$--Morse boundary. For the definition of the contracting boundary, we continue working with the contracting definition. 

\begin{definition}
	The contracting boundary of $X$, $\bdry_c X$, is defined to be the set of equivalence classes of contracting quasi-geodesic rays based at $\mathfrak{o}$. Two contracting quasi-geodesics are equivalent if they are a bounded Hausdorff distance apart. 
\end{definition}

We now introduce the topology defined in \cite{CM}. Given a sublinear function $\rho$ and constants $q\geq 1$, $Q\geq 0$, define
	\begin{equation*}
		\kappa(\rho,q,Q) = \max \set{3q, 3Q^2, 1+\inf \set{R>0 \st \forall r \geq R, 3q^2 \rho(r) <r }}. 	
	\end{equation*}
The constant $\kappa$ is defined so that it satisfies the property that for $r \geq \kappa(\rho,L,A)$,
\begin{equation*}
	r-L^2\rho(r)-A \geq  \frac{1}{3} r \geq L^2 \rho (r).
\end{equation*}
This inequality proves the following theorem \cite[Theorem 4.2]{CM}, which we will use in our proof of equivalent topologies. 
\begin{theorem}[Quasi-geodesic Image Theorem]\label{quasigeoimg}
		Let $Z$ be $\rho$-contracting. Let $\alpha:[0,T] \to X$ be a continuous $(q,Q)$-quasi-geodesics segment. If $d(\alpha,Z) \geq \kappa(\rho,q,Q)$ then
		\begin{equation*}
			\diam \pi(\alpha(0))\union \pi(\alpha(T)) \leq \frac{q^2 +1}{q^2}(Q + d(\alpha(T),Z))+ \frac{q^2-1}{q^2}d(\alpha(0),Z) + 2 \rho(d(\alpha(0),Z)).
		\end{equation*}
	\end{theorem}
	\begin{figure}
	\includegraphics[width=\textwidth]{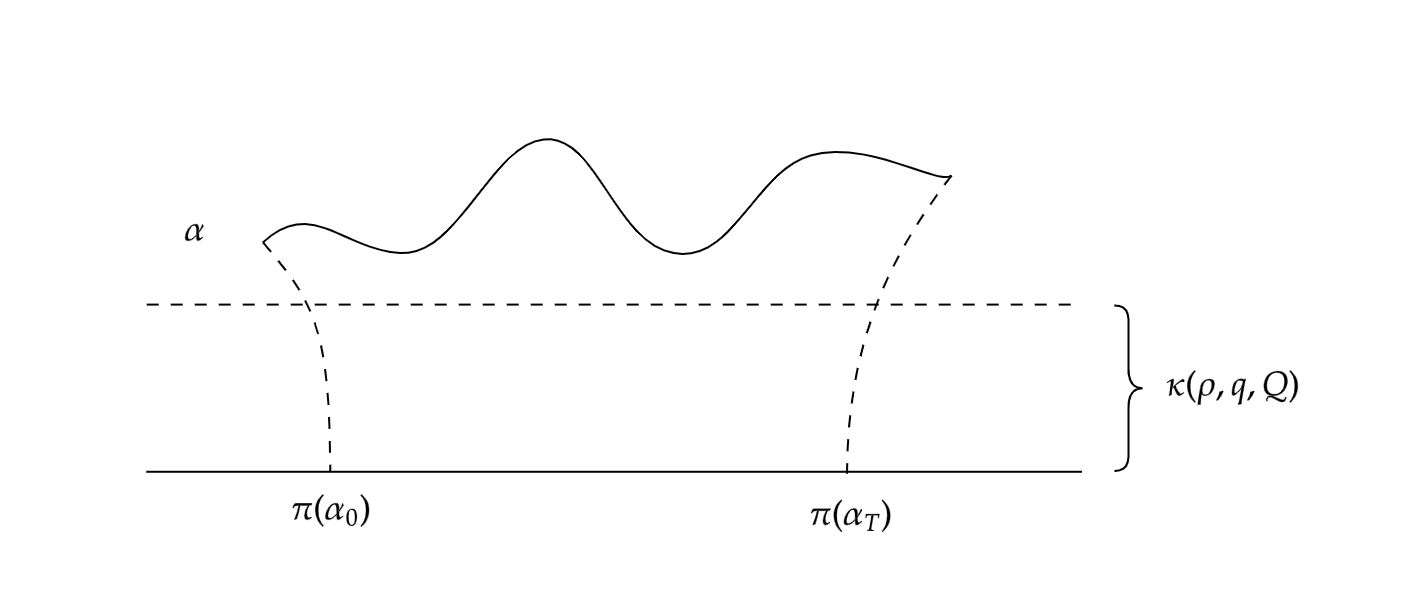}
	\caption{Quasi-geodesic Image Theorem}\label{qgImage}
	\end{figure}

 Let $\mathbf{b} \in \bdry_c X$ and let $b \in \mathbf{b}$ be the unique geodesic in the equivalence class. Let $\rho_b$ be a sublinear function such that $b$ is $\rho_b$-contracting. The topology is defined by the following open neighbourhood $U(b, r)$. 

\begin{definition}
	Let $r > 0$, $U(\mathbf{b},r)$ is defined to be the set of points $\mathbf{a} \in \bdry X$ such that for all $Q \geq 1$ and $q \geq 0$ and every continuous $(q,Q)$-quasi-geodesic ray $\alpha \in \mathbf{a}$, we have 
	\begin{equation*}
			d(\alpha, b \intersect \mathcal{N}_r^c \mathfrak{o} )\leq \kappa(\rho_b,q,Q). 
	\end{equation*}
	Here, $\mathcal{N}_r \mathfrak{o}$ stands for the $r$ neighbourhood of $\mathfrak{o}$. 
\end{definition}

\section{$\kappa$--Morse Boundary}
In this section, we will introduce the $\kappa$--Morse boundary and its topology defined in \cite{QRT1} and \cite{QRT2}. We will then check that that the 1-Morse boundary is equal to the contracting boundary as a set.

For any point $p \in X$, we use $||p||$ to denote $d_X(\mathfrak{o},p)$. Given a quasi-geodesic ray $\alpha$ starting at $\mathfrak{o}$, let $t_r$ be the first time that $||\alpha(t_r)|| = r$. We use $\alpha_r$ to denote $\alpha(t_r)$, and $\alpha|_r$ to denote $\alpha([0,t_r])$.

Given a sublinear function $\kappa$, the $\kappa$--Morse boundary introduced in \cite{QRT1} and \cite{QRT2} is attained by relaxing the Morse set. We first loosen the definition of a neighbourhood to allow a $\kappa$ multiplicative error:
\begin{equation*}
	\mathcal{N}_{\kappa}(Z,m):= \set{x\in X : d_X(x,Z) \leq m \cdot \kappa (||x||)}.
\end{equation*}
When $\kappa \equiv 1$ this is just the usual $m$-neighbourhood. 

\begin{definition}\label{weaklymorse}
We say $Z$ is $\kappa$--Morse if there exists a proper function $m_Z: [1,\infty) \times [0,\infty) \to \R$ such that for every $q \geq 1$ and every $Q \geq 0$, every $(q,Q)$-quasi-geodesic $\beta:[s,t] \to X$ with endpoints on $Z$ satisfies
\begin{equation*}
	\beta[s,t] \subset \mathcal{N}_\kappa (Z, m_Z(q,Q)). 
\end{equation*}
\end{definition}
When $\kappa \equiv 1$, this is equivalent to definition \ref{morse}. For the topology, we will work with the following definition of $\kappa$--Morse (sometimes called strongly Morse). When $Z$ is a quasi-geodesic, the two definitions of $\kappa$--Morse are equivalent \cite{QRT2}.

\begin{definition}\label{stronglymorse}
	Let $\kappa$ be a concave sublinear function. We say $Z$ is $\kappa$--Morse if there is a proper function $m_Z:\R^2 \to \R$ such that for any sublinear function $\kappa'$ and any $r >0$, there exists $R$ such that for any $(q,Q)$-quasi-geodesic ray $\beta$ with $m_Z(q,Q) \leq \frac{r}{2\kappa(r)}$,
	\begin{equation*}
		d_X(\beta_R, Z) \leq \kappa'(R) \qquad \implies \qquad \beta |_r \subset N_\kappa(Z,m_Z(q,Q)).
	\end{equation*}
	We call $m_Z(q,Q)$ the Morse gauge function. We will assume that $m_Z(q,Q) \geq \max(q,Q)$. 
\end{definition}

\begin{definition}\label{kappabdry}
	The $\kappa$--Morse boundary, $\bdry_\kappa X$, is defined to be the set of all equivalence classes of $\kappa$--Morse quasi-geodesic rays based at $\mathfrak{o}$. Two $\kappa$--Morse quasi-geodesic rays $\alpha$ and $\beta$ are equivalent if they sublinearly fellow travel each other:
	\begin{equation*}
		\lim_{r\to \infty} \frac{d_X(\alpha_r, \beta_r)}{r} = 0. 
	\end{equation*}
\end{definition}
From the definition we immediately have the following lemma.
\begin{lemma}\label{lem: equal sets}
	When $\kappa \equiv 1$, $\bdry_1 X = \bdry_c X$ as sets. 
\end{lemma}
\begin{proof}
When $\kappa \equiv 1$, the definition \ref{weaklymorse} of 1-Morse and the definition \ref{morse} of Morse coincide. It remains to be shown that the equivalence relations are the same. Let $\alpha$ and $\beta$ be 1-Morse quasi-geodesics with quasi-geodesic constants $(q_1,Q_1)$ and $(q_2, Q_2)$. If they are a bounded Hausdorff distance away, then 
\begin{equation*}
	\lim_{r\to \infty} \frac{d_X(\alpha_r, \beta_r)}{r} = 0. 
\end{equation*}
Conversely, if $\alpha$ and $\beta$ fellow travel each other, then the function $\kappa'(R) := d_X(\alpha_X,\beta_X)$ is sublinaer. Then for all $R>0$,
\begin{equation*}
	d_X(\beta_R,\alpha) \leq \kappa'(R).
\end{equation*}
By definition \ref{stronglymorse} for any $r\geq 2 m_{\alpha}(q_2,Q_2)$,
\begin{equation*}
	\beta |_r \subset N_1(\alpha,m_\alpha(q_2,Q_2)).
\end{equation*}
Similarly, for any $r\geq 2 m_{\beta}(q_1,Q_1)$,
\begin{equation*}
	\alpha |_r \subset N_1(\beta,m_\beta(q_1,Q_1)).
\end{equation*}
Since $r$ can be arbitrarily large, the Hausdorff distance between $\alpha$ and $\beta$ is bounded by the maximum of $(m_\alpha(q_2,Q_2)$ and $m_\beta(q_1,Q_1))$. This proves the equivalence.
\end{proof}
We now introduce the open neighbourhood in $\bdry_\kappa X$ that defines the topology in \cite{QRT2}.
 
\begin{definition}
	Let $r>0$ and $\mathbf{b}\in \bdry_\kappa X$. Let $b \in \mathbf{b}$ be a $\kappa$--Morse quasi-geodesic ray. Define $U_{\kappa}(\mathbf{b},r)$ to be the set of points $\mathbf{a}\in \bdry_\kappa X$ such that for any $(q,Q)$-quasi-geodesic ray $\alpha \in \mathbf{a}$,
\begin{equation*}
	m_b (q,Q) \leq \frac{r}{2\kappa(r)} \qquad \implies \qquad 
	\alpha|_r \subset N_\kappa (b, m_b(q,Q)). 
\end{equation*}
\end{definition}

\section{Proof of Equivalent Topologies}
We now show that $U_1$ and $U$ define equivalent topologies on $\bdry_1 X = \bdry_c X$. We will prove both directions of containment for the open neighbourhoods. Let $\mathbf{b} \in \bdry_cX$ and $r >0$.
\begin{propositionSp}\label{prop:UinU1}
	Given $U(\mathbf{b},r)$, there is $R>0$ such that $U(\mathbf{b},R) \subset U_1(\mathbf{b},r)$. 
\end{propositionSp}

\begin{proof}
Let $b$ be the unique geodesic in the class of $\mathbf{b}$. Let $K =\sup_{m_b(q,Q)\leq r/2} \kappa(\rho_b,q,Q)$, where $\rho_b$ is the sublinear function corresponding to the geodesic $b$ as in definition \ref{contracting}. $K$ is well defined since $\kappa(\rho_b,q,Q)$ is bounded when $\max\{q,Q\} \leq m_b(q,Q)\leq r/2$. In particular $K$ is sublinear with respect to $r$.

Since $b$ is 1-Morse, there exists $R$ such that for any $(q, Q)$-quasi-geodesic ray $\alpha$ with $m_b(q, Q) \leq r/2$, if
$d_X(\alpha|_R,b) \leq K$ then $\alpha|_r \subset \mathcal{N}(b,m_b(q,Q))$. So $U (\mathbf{b},R) \subset U_1(\mathbf{b},r)$. 
\end{proof}

For the other direction, we will make use of theorem \ref{quasigeoimg}.
	
\begin{propositionSp}\label{U1inU}
	 Given $\mathbf{b} \in \bdry X ,r>0$, there exists $R>0$ such that $U_1 (\mathbf{b},R) \subset U(\mathbf{b},r)$. 
\end{propositionSp}

\begin{proof}
Choose $R$ sufficiently large so that
\begin{align*}
	R>\max (r,1), \qquad 
	R - 4\sqrt{R} >r, \qquad 
	\rho_b(\sqrt{R}/2) < \sqrt{R}/2, \\
	\intertext{and } R > 4 \max_{q' \leq \sqrt{2r/3}, ~ Q' \leq 2r/3}m_b(q',Q')^2.
\end{align*}
The third inequality holds for sufficiently large $R$ because $\rho_b$ is sublinear. 

Let $b \in \mathbf{b}$ be the unique geodesic in the class. For any $\mathbf{a} \in U_1(\mathbf{b},R)$, let $\alpha \in \mathbf{a}$ be a continuous $(q,Q)$-quasi-geodesic. Recall that we can assume $\max(q,Q)<m_b(q,Q)$.

First consider the case where $m_b(q,Q) \leq \sqrt{R}/2$. In particular, 
\begin{equation*}
	m_b(q,Q) \leq \sqrt{R}/2 < R/2.
\end{equation*}
So by definition, $\alpha|_R \subset \mathcal{N}(b,m_b(q,Q))$. For contradiction, suppose that for all $r \leq t \leq R$,
\begin{equation*}
	d_X(\alpha(t),b) > \kappa(\rho_b,q,Q).
\end{equation*}
By theorem \ref{quasigeoimg},
\begin{align*}
	\diam \pi(\alpha_r)\union \pi(\alpha_R)
	&\leq \frac{q^2 +1}{q^2}(Q + d(\alpha_R,Z))+ \frac{q^2-1}{q^2}d(\alpha_r,Z) + 2 \rho_b(d(\alpha_r,Z))\\
	&\leq \frac{q^2 +1}{q^2}(Q + m_b(q,Q))+ \frac{q^2-1}{q^2}m_b(q,Q) + 2 \rho_b(m_b(q,Q))\\
	&= \frac{q^2+1}{q^2}Q + 2 m_b(q,Q) + 2 \rho_b(m_b(q,Q))\\
	&\leq  2 m_b(q,Q) + 2 m_b(q,Q) + 2 \rho_b(\sqrt{R}/2)\\
	& \leq 3 \sqrt{R}.
\end{align*}

	\begin{figure}
	\includegraphics[width=\textwidth]{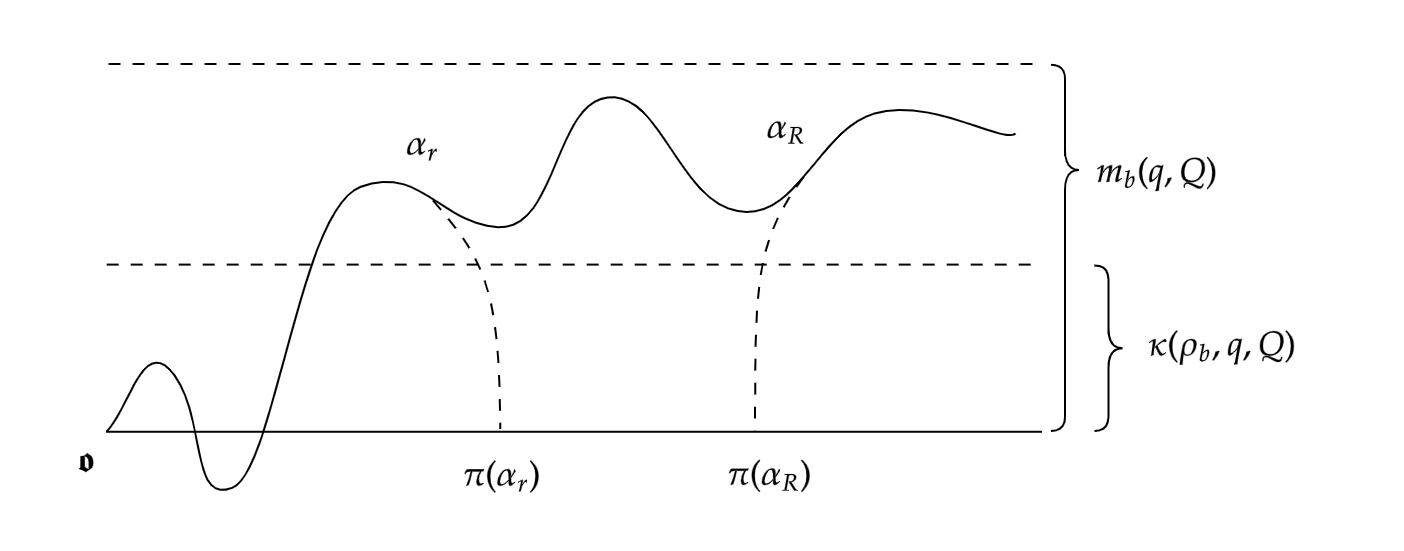}	
	\caption{Proof of proposition \ref{U1inU} in the case where $m_b(q,Q) \leq \sqrt{R}/2$.}\label{U1inUimg}
	\end{figure}

On the other hand, the projection can be bounded below by 
\begin{equation*}
	\diam \pi(\alpha_r)\union \pi(\alpha_R)
	 \geq (R-r)-d_X(\alpha_r, \pi(\alpha_r)) -d_X(\alpha_R, \pi(\alpha_R)) 
	 \geq R-r-\sqrt{R}.
\end{equation*}
Combining the two inequalities, we have that
\begin{equation*}
	3\sqrt{R} \geq R-r-\sqrt{R}.
\end{equation*}
But this contradicts the assumption that $R-4\sqrt{R}>r$. We conclude that $d_X(\alpha_t,b) \leq \kappa(\rho_b,q,Q)$ for some $r \leq t \leq R$.

We are left with the case where $m_b(q,Q) > \sqrt{R}/2$. In this case, 
\begin{equation*}
	m_b(q,Q)> \sqrt{R}/2 >\max_{q' \leq \sqrt{2r/3}, Q' \leq 2r/3} m_b(q',Q'))^2.
\end{equation*}
Then either $q> \sqrt{2r/3}$ or $Q > 2r/3$, in which case
\begin{equation*}
	\kappa(\rho_b,q,Q) \geq \max\{3Q,3q^2\} > 2r \geq d(\alpha, b \intersect N_r^c o ).
\end{equation*}
The last inequality holds because 
\begin{equation*}
	d(\alpha, b \intersect N_r^c o) \leq d(\alpha_r,o) + d(b_r, o) = 2r.
\end{equation*}

We conclude that for sufficiently large $R$, $d(\alpha, b \intersect N_r^c o) \leq \kappa(\rho_b,q,Q)$. Hence $U_1(\mathbf{b},R) \subset U(\mathbf{b},r)$.
\end{proof}

We have proven theorem \ref{main} by combining lemma \ref{lem: equal sets}, proposition \ref{prop:UinU1}, and proposition \ref{U1inU}.

\newpage
\bibliography{mybib}{}
\bibliographystyle{plain}

\end{document}